\documentclass[12pt,a4paper]{amsart}

\usepackage[latin1]{inputenc}

\usepackage{enumerate}

\usepackage{float}
\restylefloat{table}

\usepackage{graphicx}

\usepackage{hyperref}

\usepackage{amsmath}

\usepackage{amssymb}

\usepackage{amsthm}

\usepackage{tikz-cd}

\usepackage{a4wide}

\newtheoremstyle{plain}
  {6pt}   
  {6pt}   
  {\itshape}  
  {0pt}       
  {\bfseries} 
  {.}         
  {5pt plus 1pt minus 1pt} 
  {}          

\newtheoremstyle{definition}
  {6pt}   
  {6pt}   
  {\normalfont}  
  {0pt}       
  {\bfseries} 
  {.}         
  {5pt plus 1pt minus 1pt} 
  {}          

\theoremstyle{plain}
\newtheorem*{thm*}{Theorem}
\newtheorem{thm}{Theorem}[section]
\newtheorem{prop}[thm]{Proposition}

\newtheorem{lem}[thm]{Lemma}

\theoremstyle{definition}
\newtheorem{defn}[thm]{Definition}

\newtheorem{ex}[thm]{Example}
\newtheorem{rmk}[thm]{Remark}

\numberwithin{equation}{thm}

\newcommand{\emphbf}[1]{\emph{\textbf{#1}}}

\DeclareMathAlphabet{\mathpzc}{OT1}{pzc}{m}{it}

\newcommand{\rad}[1]{\radoperator(#1)}

\DeclareMathOperator{\radoperator}{rad}

\DeclareMathOperator{\Hom}{Hom}
\DeclareMathOperator{\Ext}{Ext}

\DeclareMathOperator{\modu}{mod}

\DeclareMathOperator{\End}{End}

\DeclareMathOperator{\image}{Im}

\DeclareMathOperator{\op}{op}

\newcommand\ddfrac[2]{\frac{\displaystyle #1}{\displaystyle #2}}

\begin{document}

\title[Bocses, quasi-hereditary algebras \& Borel subalgebras]{From quasi-hereditary algebras with exact Borel subalgebras to directed bocses}
\author{Tomasz Brzezi\'nski}
\author{Steffen Koenig} 
\author{Julian K\"ulshammer}
\date{\today}

\address{Tomasz Brzezi\'nski\\
Department of Mathematics, Swansea University\\
Swansea University Bay Campus\\
Fabian Way\\
  Swansea SA1 8EN, U.K.\ \newline \indent
Department of Mathematics, University of Bia{\l}ystok\\ K.\ Cio{\l}kowskiego  1M\\
15-245 Bia\-{\l}ys\-tok, Poland}
\email{T.Brzezinski@swansea.ac.uk}

\address{Steffen Koenig\\
Institute of Algebra and Number Theory,
University of Stuttgart \\ Pfaffenwaldring 57 \\ 70569 Stuttgart,
Germany} \email{skoenig@mathematik.uni-stuttgart.de}

\address{
Julian K\"ulshammer\\
Department of Mathematics, Uppsala University\\
Box 480 \\ 75106 Uppsala, Sweden} \email{julian.kuelshammer@math.uu.se}

\begin{abstract}
Up to Morita equivalence, every quasi-hereditary algebra is the dual algebra 
of a directed bocs or coring. From the bocs, an exact Borel subalgebra is
obtained. In this paper a characterisation of exact Borel subalgebras arising 
in this way is given. 
\end{abstract}

\maketitle

\section{Introduction}

The theorem of Poincar\'e, Birkhoff and Witt is a
fundamental result in the representation theory of semisimple complex Lie
algebras $\mathfrak g$. When
$\mathfrak b$ is a Borel subalgebra of $\mathfrak g$ and $\mathfrak h$ a
Cartan subalgebra, then by the PBW theorem induction from $\mathfrak b$ to 
$\mathfrak g$ is an exact functor. In particular, inducing up simple 
$\mathfrak b$-modules (which coincide with the simple $\mathfrak h$-modules)
yields the universal highest weight modules, the Verma modules. The relevant
categorical setup is the Bernstein--Gelfand--Gelfand category $\mathcal O$, 
which decomposes into a direct sum of blocks. Each block is equivalent to
the module category of a quasi-hereditary algebra, whose standard modules
correspond to the Verma modules in the block. 

An analogue of the PBW theorem has been established in \cite{KKO14} for
quasi-hereditary algebras in general, thus also covering Schur algebras of
reductive algebraic groups, hereditary algebras and algebras of global
dimension two, among many others. More precisely, it has been shown that
up to Morita equivalence every quasi-hereditary algebra has an exact Borel
subalgebra in the sense of \cite{Koe95}.  
The construction given in \cite{KKO14} is based
on describing the exact category of standardly filtered modules of a
quasi-hereditary algebra as the category of representations 
of a certain bocs 
(i.e.\ a ``bimodule over a category with coalgebra structure'' \cite{Roi79}) or coring (i.e.\ a comonoid in the monoidal category of bimodules \cite{Swe75}).
This characterises quasi-hereditary algebras. The exact
Borel subalgebras obtained by this construction have strong additional
properties, which are not present in other such algebras constructed in
particular situations such as for category $\mathcal O$ (in \cite{Koe95}). 
In particular, induction preserves significant parts of 
cohomology and Ringel duality becomes a
construction on the directed bocs. 

The aim of this article is to analyse and clarify the stronger properties of 
the exact Borel subalgebras arising from the construction in \cite{KKO14}, 
and to show that these properties are present exactly when the 
quasi-hereditary algebra $\Lambda$ and the given 
exact Borel subalgebra $B$
correspond to a directed bocs as in \cite{ KKO14}. More
precisely, to the ring extension $B \subseteq \Lambda$ we associate the $B$-coring 
$\Hom_{B^{\op}}(\Lambda,B)$ and
translate properties of this ring extension into properties of the corresponding coring. Subsequently we show that directed bocses necessarily correspond in this way to quasi-hereditary algebras with
 exact Borel subalgebras.
The first step,
Theorem \ref{surjectivecounit}, translates properties such as $\Lambda$ being 
a progenerator as right $B$-module into surjectivity of the counit of the $B$-coring
$\Hom_{B^{\op}}(\Lambda,B)$ and the 
existence of a group-like element. The second step, 
Theorem \ref{kernelprojectivising}, relates
invariance of cohomology under induction with the kernel of the counit being
projectivising. The main result, Theorem \ref{maintheorem}, 
then gives the desired converse
of the main result of \cite{KKO14}, showing that the strong properties of
the exact Borel subalgebra $B$ imply that $B \subseteq \Lambda$ comes from a bocs
as in \cite{KKO14}.
\bigskip

{\em Notation.}
For a finite dimensional algebra $B$, the category of left $B$-modules is
denoted by $\modu B$, the category of right $B$-modules by $\modu B^{\op}$. A 
complete set of representatives of the isomorphism classes of simple left 
$B$-modules is denoted by $L_B(\mathtt{1}),\dots,L_B(\mathtt{n})$. 
The projective cover and injective hull of $L_B(\mathtt{i})$ are 
$P_B(\mathtt{i})$ and $I_B(\mathtt{i})$, respectively. 

\section{Rings and corings}

In this section we recall the necessary terminology for bocses and corings and describe how certain properties of ring extensions translate to the language of corings by taking the dual. 

\begin{defn}
Let $B$ be an algebra. A \emphbf{$B$-coring}  is a $B$-$B$-bimodule $W$ with a $B$-$B$-bilinear coassociative comultiplication $\mu \colon W\to W\otimes_B W$ and a $B$-$B$-bilinear counit $\varepsilon\colon W\to B$, i.e. the following diagrams commute:
\[\begin{tikzcd}
W\arrow{r}{\mu}\arrow{d}{\mu} &W\otimes_B W\arrow{d}{\mu\otimes 1}\\
W\otimes_B W\arrow{r}{1\otimes \mu} &W\otimes_BW\otimes_B W
\end{tikzcd} \text{ and } 
\begin{tikzcd}
W\otimes_B B\arrow{dr}[swap]{\cong} &W\otimes_B W\arrow{l}[swap]{1\otimes \varepsilon} \arrow{r}{\varepsilon\otimes 1} &B\otimes_B W.\arrow{dl}{\cong}\\
&W\arrow{u}{\mu}
\end{tikzcd}
\]
The pair $(B,W)$ is called a  \emphbf{bocs}.\footnote{Although it might not be a generally accepted practice, in order to stay aligned with the existing literature both on representation theory and on Hopf algebras, we use both terms ``coring'' and ``bocs'' in this text; coring  refers to a bimodule (with a coassociative and counital comultiplication) over a fixed algebra while  bocs refers to a pair: algebra, bimodule (with a coassociative and counital comultiplication).}
\end{defn}

The reader is referred to \cite{BW03} and \cite{BSZ09} for  comprehensive treatments of corings and bocses. Bocses will arise in this article as duals of ring extensions. The following proposition is well-known. A slightly more general statement can e.g. be found in \cite[Theorem 1]{Kle84}.

\begin{prop}
Let $B\subseteq \Lambda$ be a subring of a ring $\Lambda$ such that $\Lambda$ is finitely generated projective as a right $B$-module. Let $W:=\Hom_{B^{\op}}(\Lambda,B)$, and let 
\[\psi\colon W\otimes_B W\longrightarrow \Hom_{B^{\op}}(\Lambda\otimes_B \Lambda,B)\] 
be the map given by $f\otimes g\mapsto f\otimes g$. Then, the map $\psi$ is bijective and $W$  is finitely generated projective, when considered as a left $B$-module. 

Furthermore, $W$ has the structure of a $B$-coring with counit given by the map 
\[\varepsilon\colon \Hom_{B^{\op}}(\Lambda,B)\longrightarrow B,\qquad f\longmapsto f(1),\]
and comultiplication given by the unique map $\mu\colon W\to W\otimes_B W$ rendering commutative the following diagram: 
\[
\begin{tikzcd}
W\arrow{d}{\mu}\arrow[equals]{r}&\Hom_{B^{\op}}(\Lambda,B)\arrow{d}{\Hom_{B^ {\op}}(m,B)}\\
W\otimes_B W&\Hom_{B^{\op}}(\Lambda\otimes_B \Lambda,B)\arrow{l}[swap]{\psi^{-1}}\, ,
\end{tikzcd}
\] where $m$ is the multiplication map $\Lambda\otimes_B \Lambda\to \Lambda$.
\end{prop}

In the theory of corings and bocses, two basic properties are that the counit 
is surjective, and that the bocs is normal:

\begin{defn}
A bocs $(B,W)$ is called \emphbf{normal} if there exists a \emphbf{group-like} element in the $B$-coring $W$, i.e. $\omega\in W$ such that $\mu(\omega)=\omega\otimes \omega$ and $\varepsilon(\omega)=1$. 
\end{defn}

\begin{rmk}
\begin{itemize}
\item It is easy to see that normality implies that the counit is surjective. The converse is not true in general. 
\item If $B$ has no (left or right) zero divisors  and $\omega\neq 0$, then  $\varepsilon(\omega)=1$ in fact follows from $\mu(\omega)=\omega\otimes\omega$, since the counitality implies that $\varepsilon(\omega)$ is an idempotent in $B$. 
\end{itemize}
\end{rmk}

Our first result answers the question of when a ring extension gives rise to a bocs satisfying these special properties.

\begin{thm}\label{surjectivecounit}
Let $B\subseteq \Lambda$ be a ring extension. Assume that $\Lambda$ is finitely generated and projective as a right $B$-module. 
\begin{enumerate}[(i)]
\item The following statements are equivalent:
\begin{enumerate}[(1)]
\item The inclusion $\iota\colon B\hookrightarrow \Lambda$ splits as a map of right $B$-modules.
\item $\Lambda/\iota(B)$ is projective as a right $B$-module.
\item $\Lambda$ is a projective generator as a right $B$-module. 
\item The counit of the right dual coring $\Hom_{B^{\op}}(\Lambda,B)$ is surjective.  
\end{enumerate}
\item The following statements are equivalent:
\begin{enumerate}[(1)]
\item There is a splitting of the inclusion $\iota\colon B\hookrightarrow \Lambda$ as right $B$-modules whose kernel is a right ideal of $\Lambda$. 
\item The right dual coring $\Hom_{B^{\op}}(\Lambda,B)$ is normal. 
\end{enumerate}
\end{enumerate}
\end{thm}

\begin{proof}
\begin{enumerate}[(i)]
\item We start by showing that (1) and (2) are equivalent. Assume that (1) holds, i.e. that the exact sequence $0\to B\xrightarrow{\iota} \Lambda\to \Lambda/\iota(B)\to 0$ splits. In particular, $\Lambda/\iota(B)$ is a direct summand of the projective right $B$-module $\Lambda$ and thus is itself projective. Therefore (2) holds. Conversely, the projectivity of $\Lambda/\iota(B)$ implies that the sequence splits in the category of right $B$-modules.

That (1) implies (3) follows from the fact that $\Lambda$ is a generator as a 
right $B$-module, since $B$ is a direct summand of $\Lambda$.

The direction that (3) implies (4) uses parts of Morita theory. Recall that an equivalent formulation of $\Lambda$ being a generator over $B$ is that the trace of $\Lambda$ in $B$ is $B$, i.e. that $B=\sum_{g\in \Hom_{B^{\op}}(\Lambda,B)} g(\Lambda)$, see e.g. \cite[Theorem 18.8]{Lam99}. In particular, for every $b\in B$, there exist $a_i\in \Lambda$ and $g_i\in \Hom_{B^{\op}}(\Lambda,B)$ such that 
\[b=\sum_i g_i(a_i)=\sum_i g_i(\lambda_{a_i}(1))\]
where $\lambda_{a_i}$ denotes the right $B$-module endomorphism of $\Lambda$ given by left  multiplication with $a_i$. In particular, $b=\varepsilon(\sum_i g_i\circ \lambda_{a_i})$, and therefore the counit is surjective. 

For the remaining direction (4) implies (1) observe that $\varepsilon$ is surjective if and only if $\Hom_{B^{\op}}(\iota,B)$ is surjective as the following diagram commutes
\[
\begin{tikzcd}[column sep=10ex]
W\arrow{r}{\varepsilon}&B\\
\Hom_{B^{\op}}(\Lambda,B)\arrow[equals]{u}\arrow{r}{\Hom_{B^{\op}}(\iota,B)} &\Hom_{B^{\op}}(B,B).\arrow{u}{\cong}
\end{tikzcd}
\]
Consider the exact sequence of right $B$-modules $0\to B\xrightarrow{\iota} \Lambda\to \Lambda/\iota(B)\to 0$. Applying $\Hom_{B^{\op}}(-,B)$ yields the long exact sequence 
\begin{align*}
0\longrightarrow \Hom_{B^{\op}}(\Lambda/\iota(B),B)\longrightarrow &\Hom_{B^{\op}}(\Lambda,B)\xrightarrow{\Hom_{B^{\op}}(\iota,B)} \Hom_{B^{\op}}(B,B)\\
&\longrightarrow \Ext^1_{B^{\op}}(\Lambda/\iota(B),B)\to \Ext^1_{B^{\op}}(\Lambda,B).
\end{align*}
The last term of this sequence vanishes as $\Lambda$ is projective as a right $B$-module. Moreover, $\Hom_{B^{\op}}(\iota,B)$ is surjective. Therefore, $\Ext^1_{B^{\op}}(\Lambda/\iota(B),B)=0$ and the exact sequence of right $B$-modules $0\to B\xrightarrow{\iota} \Lambda\to \Lambda/\iota(B)\to 0$ splits.
\item This is proved in \cite[Theorem 3]{Kle84}. \qedhere
\end{enumerate}
\end{proof}

\begin{ex}
For illustration, we provide non-instances of the preceding theorem. 
\begin{enumerate}[(i)]
\item Let $\Lambda=M_2(\Bbbk)$ and let $B$ be the subalgebra of upper triangular matrices. Let $e_\mathtt{2}=\left(\begin{smallmatrix}1&0\\0&0\end{smallmatrix}\right)\in B$. Then 
\[
\Lambda_B=\begin{pmatrix}1&0\\0&0\end{pmatrix}B\oplus \begin{pmatrix}0&0\\1&0\end{pmatrix}B\cong (e_\mathtt{2}B)^{\oplus 2}
\] 
is finitely generated and projective as a right $B$-module but not a projective generator. Thus, the counit of the corresponding coring is not surjective. 
\item Let $\Lambda=M_2(\Bbbk)$ and let $B$ be the subalgebra of diagonal matrices. In this case the algebra $\Lambda$ is a projective generator over $B$. However, the corresponding coring (described in \cite[p. 308]{Roi79}) is not normal. 
\end{enumerate}
\end{ex}

Continuing the discussion on how to translate properties of bocses to properties of ring extensions we consider the property that the kernel of the counit $\varepsilon \colon W\to B$ is a projective bimodule (and the slightly weaker property of being projectivising in the sense of \cite{BB91}). We slightly altered the terminology in \cite{BB91} by having a separate terminology for the two properties.

\begin{defn}
Let $B$ be a ring. A $B$-$B$-bimodule $U$ is called \emphbf{left projectivising} if $U\otimes_B X$ is projective as a left $B$-module for each $X\in \modu B$. Symmetrically, a bimodule is called \emphbf{right projectivising} if $X\otimes_B U$ is projective as a right $B$-module for each $X\in \modu B^{\op}$. A bimodule $U$ is called \emphbf{projectivising} if it is both left and right projectivising.   
\end{defn}

Any $B$-$B$-bimodule of the form $\bigoplus_{i,j} P_i\otimes_\Bbbk Q_j$ for a projective left module $P_i$ and a projective right module $Q_j$ is both left and right projectivising. 
Over a perfect field, a $B$-$B$-bimodule is projective if and only if it is isomorphic to a direct summand of a module of the above form. In particular, it is left and right projectivising. In this case, the following proposition shows that also the converse holds.

\begin{prop}[{\cite[Theorem 3.1]{AR91b}}]\label{perfect}
Let $B$ be a finite dimensional algebra over a perfect field $\Bbbk$ and let $W$ be a $B$-$B$-bimodule such that $W$ is projective as a right module and $W\otimes_B (B/\rad{B})$ is projective as a left module. Then, $W$ is projective as a $B$-$B$-bimodule. In particular, a right projective and left projectivising bimodule is projective as a bimodule.
\end{prop}

The next statement follows from the proof of \cite[Lemma 3.6]{BB91}. For convenience of the reader we include the proof.

\begin{lem}\label{injectivising}
Let $B$ be a ring. A $B$-$B$-bimodule $U$ that is projective as a right $B$-module is left projectivising if and only if $\Hom_B(U,X)$ is injective for every left $B$-module $X$.
\end{lem}

\begin{proof}
Let $0\to M_1\to M_2\to M_3\to 0$ be an exact sequence of left $B$-modules. Since $U$ is projective as a right $B$-module, also the sequence 
\begin{equation}\label{equation1}
0\to U\otimes_B M_1\to U\otimes_B M_2\to U\otimes_B M_3\to 0
\end{equation}
is exact. Let $X$ be a left $B$-module.  Applying the functor $\Hom_B(-,X)$ to the above exact sequence yields the exact sequence
\begin{equation}\label{equation2}
\begin{aligned}
0\longrightarrow\Hom_B(U\otimes_B M_3,X)&\longrightarrow \Hom_B(U\otimes_B M_2,X) \\ &\longrightarrow \Hom_B(U\otimes_B M_1,X)
\longrightarrow \Ext^1_B(U\otimes_B M_3,X).
\end{aligned}
\end{equation}
Similarly, applying $\Hom_B(-,\Hom_B(U,X))$ to the given exact sequence yields the exact sequence
\begin{equation}\label{equation3}
\begin{aligned}
0\to \Hom_B(M_3,\Hom_B(U,X)) &\to \Hom_B(M_2,\Hom_B(U,X))\\ &\to \Hom_B(M_1,\Hom_B(U,X))\to  \Ext^1_B(M_3,\Hom_B(U,X)).
\end{aligned}
\end{equation}
Adjointness of tensor and Hom provides natural identifications of
the three Hom-spaces in the exact sequence \eqref{equation2} with those in
the exact sequence \eqref{equation3}. When $U$ is left projectivising, 
each sequence \eqref{equation1} and hence also \eqref{equation2} is split exact.
Therefore the corresponding part of the sequence  \eqref{equation3} is split
exact, too, and $\Hom_B(U,X)$ is injective for every left $B$-module $X$.
Conversely, injectivity of $\Hom_B(U,X)$ shows that \eqref{equation3} starts
with a split exact sequence and thus $U$ is left projectivising.
\end{proof}

Before stating which property of a subalgebra corresponds to the kernel of the counit of the dual coring being projectivising we recall the notion of a right algebra from \cite[(1.1')]{BB91}. 

\begin{defn}\label{def.right.alg}
Let $(B,W)$ be a bocs. The \emphbf{right algebra} $R$ of $W$ is the algebra $\Hom_B(W,B)^{\op}$. Here, the product of $s,t\in \Hom_B(W,B)$ is given as the composition of the following maps 
\[W\xrightarrow{\mu} W\otimes_B W\xrightarrow{1\otimes s} W\otimes_B B\xrightarrow{\cong} W\xrightarrow{t} B,\]
and the identity is $\varepsilon$.
\end{defn}

\begin{rmk}
As mentioned, the terminology is from the theory of bocses. The literature on corings calls this algebra the opposite algebra of the left dual algebra, see e.g. \cite{BW03}.  That $\Hom_B(W,B)$ is an algebra for any $B$-coring $W$ as in Definition~\ref{def.right.alg} was observed by Sweedler already in \cite[3.2 Proposition]{Swe75}. The corresponding structure on $\Hom_{B^{\op}}(W,B)$ is called the left algebra in the bocs literature or the right dual algebra in the coring literature.
\end{rmk}

The following result from \cite[(3.8)]{BB91} translates the property of a bocs to have a surjective counit to the language of ring extensions. 

\begin{lem}\label{lem.proj.ext}
Let $(B,W)$ be a bocs with surjective counit. Let $\overline{W}:=\ker\varepsilon$ be left projectivising and right projective. Let $R$ be its right algebra. Then, the homomorphisms $\Ext^i_B(M,N)\to \Ext^i_R(R\otimes_B M,R\otimes_B N)$ induced by the induction functor are epimorphisms for $i= 1$ and isomorphisms for $i\geq 2$.
\end{lem}

The main result of this section is the converse to Lemma~\ref{lem.proj.ext}. The proofs of the two 
results follow the same lines. 

\begin{thm}\label{kernelprojectivising}
Let $B\subseteq \Lambda$ be a subalgebra such that $\Lambda$ is a projective generator for $B$ as a right $B$-module. Assume furthermore, that the homomorphisms 
\[\Ext^i_B(M,N)\to \Ext^i_\Lambda(\Lambda\otimes_B M,\Lambda\otimes_B N)\]
induced by the induction functor are epimorphisms for $i\geq 1$ and isomorphisms for $i\geq 2$. Let $W=\Hom_{B^{\op}}(\Lambda,B)$ be the right dual coring for $\Lambda$ and define $\overline{W}:=\ker\varepsilon$. Then $\overline{W}$ is left projectivising.  
\end{thm}

\begin{proof}
By Theorem \ref{surjectivecounit}, $W=\Hom_{B^{\op}}(\Lambda,B)$ is a $B$-coring  with surjective counit. Hence, there is a short exact sequence of $B$-$B$-bimodules
\[0\longrightarrow \overline{W}\longrightarrow W\longrightarrow B\longrightarrow 0,\] 
which splits as a sequence of left modules. Applying $\Hom_B(-,N)$ for a left $B$-module $N$ thus gives a short exact sequence
\begin{equation}\label{equation4}
0\longrightarrow N\longrightarrow\Hom_B(W,N)\longrightarrow \Hom_B(\overline{W},N)\longrightarrow 0\, .
\end{equation}
As $\Lambda$ is finitely generated projective as a right $B$-module, there is an isomorphism 
\[\Hom_B(W,N)\cong \Lambda\otimes_B N.\]
Furthermore, by the Eckmann--Shapiro Lemma, see e.g. \cite[(3.1)]{BB91}, the same assumption yields isomorphisms
\[\Ext^i_B(M,\Lambda\otimes_B N)\cong \Ext^i_\Lambda(\Lambda\otimes_B M,\Lambda\otimes_B N),\]
for all left $B$-modules $M$,$N$.
Applying $\Hom_B(M,-)$ to the exact sequence \eqref{equation4}, using these isomorphisms and denoting $\Hom_B(\overline{W},N)$ by $I$ yields the long exact sequence
\[\dots\longrightarrow\Ext^i_B(M,N)\longrightarrow \Ext^i_\Lambda(\Lambda\otimes_B M,\Lambda\otimes_B N)\longrightarrow \Ext^i_B(M,I)\longrightarrow \dots\]
The assumed epimorphism for $i=1$ and isomorphisms for $i\geq 2$ then imply 
$\Ext^i_B(M,I)=0$, for all $i\geq 1$. Hence, $I=\Hom_B(\overline{W},N)$ is an injective module for all left $B$-modules $N$. By Lemma \ref{injectivising}, this is equivalent to $\ker\varepsilon$ being left projectivising.
\end{proof}

We finish this section by providing an example of a non-instance of the preceding theorem from \cite[Appendix A.4]{KKO14}. 

\begin{ex}\label{ex.quiver}
Let $\Lambda$ be the path algebra of the quiver
\[
\begin{tikzcd}
\mathtt{2}\arrow[yshift=0.5ex]{rr}{\alpha}&&\mathtt{3}\arrow[yshift=-0.5ex]{ll}{\beta}\arrow{ld}{\delta}\\
&\mathtt{1}\arrow{lu}{\gamma}
\end{tikzcd}
\]
with relations $\gamma\delta=0=\alpha\beta$. Then $\Lambda$ has a subalgebra $B$ given by the subquiver with arrows $\alpha$ and $\gamma$, over which it is a projective generator as a right $B$-module. Then ${0=\Ext^1_B(L(\mathtt{1}),L(\mathtt{3}))}$ but $\dim \Ext^1_B(\Lambda\otimes_B L(\mathtt{1}),\Lambda\otimes_B L(\mathtt{3}))=1$. 
\end{ex}

\section{Homological exact Borel subalgebras}

In this section the results obtained in the previous section are applied to the setting of quasi-hereditary algebras and exact Borel subalgebras. For further reading on quasi-hereditary algebras we suggest the original articles \cite{Sco87, CPS88} and the survey articles \cite{DR92, KK99}. We start by recalling these notions. Throughout this section assume that $\Bbbk$ is a splitting field for $\Lambda$. 

\begin{defn}
A finite dimensional algebra $\Lambda$ is \emphbf{quasi-hereditary} if there exist modules $\Delta(\mathtt{i})$, $\mathtt{i}\in \{\mathtt{1},\dots,\mathtt{n}\}$ satisfying 
\begin{enumerate}[(QH1)]
\item $\End_\Lambda(\Delta(\mathtt{i}))\cong \Bbbk$.
\item $\Hom_\Lambda(\Delta(\mathtt{i}),\Delta(\mathtt{j}))\neq 0\Rightarrow \mathtt{i}\leq \mathtt{j}$. 
\item $\Ext^1_\Lambda(\Delta(\mathtt{i}),\Delta(\mathtt{j}))\neq 0\Rightarrow \mathtt{i}<\mathtt{j}$.
\item ${}_\Lambda\Lambda\in \mathcal{F}(\Delta)$, where $\mathcal{F}(\Delta)$ denotes the subcategory of all $\Lambda$-modules which can be filtered by the $\Delta(\mathtt{i})$, i.e. 
\[\mathcal{F}(\Delta)=\{M\in \modu \Lambda\,|\, \exists 0=M_0\subseteq M_1\subseteq \dots \subseteq M_s\text{ such that } \forall t \,\, \exists \mathtt{i}_t: M_t/M_{t+1}\cong \Delta(\mathtt{i}_t) \}.\]
\end{enumerate}
The modules $\Delta(\mathtt{i})$ are called the \emphbf{standard modules}. 
\end{defn}

\begin{ex}\label{examplequasihereditary}
Examples of quasi-hereditary algebras include blocks of the Bernstein--Gelfand--Gelfand category $\mathcal{O}$, classical and quantised Schur algebras of reductive algebraic groups and
of symmetric groups and of Brauer algebras, and algebras of global dimension smaller than or equal to two. For later use, we give an example of a monomial quasi-hereditary algebra. It appeared in a previous version of \cite{GS17}: Let $\Lambda$ be the path algebra of the quiver
\[\begin{tikzcd}
\mathtt{3}\arrow{r}{a}\arrow{d}{c}&\mathtt{4}\arrow{d}{b}\\
\mathtt{1}\arrow{r}{d}&\mathtt{2}\arrow{ul}[swap]{e}
\end{tikzcd}
\]
modulo the relations $dc-ba, ae, eb$. Its standard modules are given by $\Delta(\mathtt{1})=L(\mathtt{1})$, $\Delta(\mathtt{2})=L(\mathtt{2})$, $\Delta(\mathtt{3})=\begin{matrix}\mathtt{3}\\\mathtt{1}\end{matrix}$, and $\Delta(\mathtt{4})=\begin{matrix}\mathtt{4}\\\mathtt{2}\end{matrix}$. 
For further examples, see e.g. \cite{DR92, KK99}. 
\end{ex}

\begin{rmk}\label{costandard}
\begin{enumerate}[(i)]
\item Being quasi-hereditary can also be defined using a dual condition: An algebra is quasi-hereditary if and only if there exist modules $\nabla(\mathtt{i})$, $\mathtt{i}\in \{\mathtt{1},\dots,\mathtt{n}\}$ satisfying 
\begin{enumerate}[(QH1')]
\item $\End_\Lambda(\nabla(\mathtt{i}))\cong \Bbbk$.
\item $\Hom_\Lambda(\nabla(\mathtt{i}),\nabla(\mathtt{j}))\neq 0\Rightarrow \mathtt{i}\geq \mathtt{j}$.
\item $\Ext^1_\Lambda(\nabla(\mathtt{i}),\nabla(\mathtt{j}))\neq 0\Rightarrow \mathtt{i}>\mathtt{j}$.
\item $D\Lambda\in \mathcal{F}(\nabla)$. 
\end{enumerate}
\item The $\Delta(\mathtt{i})$ and $\nabla(\mathtt{i})$ are uniquely determined (given the implicit order on the simple modules). They can be defined as 
\[
\Delta(\mathtt{i})=\ddfrac{P(\mathtt{i})}{\sum_{\substack{g\colon P(\mathtt{j})\to P(\mathtt{i})\\\mathtt{j}>\mathtt{i}}} \image g}
\] 
and $\nabla(\mathtt{i})$ is the maximal submodule of $I(\mathtt{i})$ all of whose composition factors are of the form $L(\mathtt{j})$ for $\mathtt{j}\leq \mathtt{i}$. To check that an algebra is quasi-hereditary, one can also check that (QH1) and (QH4) (or (QH1') and 
(QH4')) 
are satisfied for these explicitly defined modules. 
\end{enumerate}
\end{rmk}

We recall the notion of an exact Borel subalgebra from \cite{Koe95} and provide the new notion of a homological
exact Borel subalgebra which will be precisely the class of exact Borel subalgebras giving rise to directed corings.

\begin{defn}
  Let $\Lambda$ be a quasi-hereditary algebra. A subalgebra $B\subseteq \Lambda$ is called an \emphbf{exact Borel subalgebra} if its isomorphism classes of simple modules can be indexed by the same indexing set $\{\mathtt{1},\dots,\mathtt{n}\}$ as the standard modules for $\Lambda$  and
\begin{itemize}
\item[(B1)] the algebra $\Lambda$ is directed, i.e. it is quasi-hereditary with simple standard modules,
\item[(B2)] the algebra $\Lambda$ is projective when considered as a right $B$-module,
\item[(B3)] the standard modules for $\Lambda$ can be obtained as $\Delta_\Lambda(\mathtt{i})=\Lambda\otimes_B L_B(\mathtt{i})$.
\end{itemize}
\begin{itemize}
\item[(H)  ] If the subalgebra $B$ furthermore  has the property that
the homomorphisms 
\[
\Ext^i_B(M,N)\longrightarrow \Ext^i_\Lambda(\Lambda\otimes_B M,\Lambda\otimes_B N)
\] 
induced by the induction functor are epimorphisms for $i\geq 1$ and isomorphisms for $i\geq 2$, then $B$ is called a \emphbf{homological
  exact Borel subalgebra}.
\item[(N)  ]
  An exact Borel subalgebra is said to be \emphbf{normal} if there is a splitting of the inclusion $\iota\colon B\hookrightarrow \Lambda$ as right $B$-modules whose kernel is a right ideal of $\Lambda$.
\item[(R)  ]
  A normal exact Borel subalgebra is said to be \emphbf{regular}
  if there are isomorphisms
$
\Ext^i_B(L_B(\mathtt{j}),L_B(\mathtt{k}))\longrightarrow \Ext^i_\Lambda(\Lambda\otimes_B L_B(\mathtt{j}),\Lambda\otimes_B L_B(\mathtt{k}))
$ 
for all $i\geq 1$ and all $\mathtt{j},\mathtt{k}$ induced by the induction functor. 
\end{itemize}
\end{defn} 

\begin{rmk}
Note that for a normal exact Borel subalgebra being regular implies being homological as can be seen from the long exact sequence of $\Ext$-groups. 
\end{rmk}

The main result of \cite{KKO14} proves the existence of a regular homological exact Borel subalgebra of every quasi-hereditary algebra up to Morita equivalence. The following example gives one particular instance of this result:

\begin{ex}
For the algebra in Example \ref{examplequasihereditary} the method given in \cite{KKO14} produces an exact Borel subalgebra of $\End(\Lambda\oplus P(\mathtt{4}))$ which is isomorphic to the path algebra of the quiver
\[
\begin{tikzcd}
\mathtt{3}\arrow{d}{d'}&\mathtt{2}\arrow{l}[swap]{f'}\\
\mathtt{4}&\mathtt{1}\arrow{u}{a'}\arrow{l}[swap]{e'}
\end{tikzcd}
\]
with relation $d'f'$. More examples of exact Borel subalgebras can be found in \cite{Koe95, KKO14, BK18}. 
\end{ex}

The following lemma is essential for applying Theorem \ref{surjectivecounit} to deduce that the counit of the dual coring is surjective. 

\begin{lem}\label{projectivegenerator}
Let $\Lambda$ be a quasi-hereditary algebra with exact Borel subalgebra $B$. Then $\Lambda$ is a projective generator for $B$ as a right $B$-module.
\end{lem}

\begin{proof}
By \cite[Theorem A]{Koe95}, for an exact Borel subalgebra of a quasi-hereditary algebra there is an isomorphism of $B$-modules $\nabla_B(\mathtt{i})\cong \nabla_\Lambda(\mathtt{i})$. By the dual definition of quasi-hereditary using costandard modules, see Remark \ref{costandard}, $D\Lambda$ is filtered by $\nabla_\Lambda(\mathtt{i})$ with all $\mathtt{i}$ appearing at least once. Since $B$ is a quasi-hereditary algebra with simple standard modules, 
the costandard modules $\nabla_B(\mathtt{i})$ are injective. Therefore the filtration of $D\Lambda$ by costandard modules as a left $B$-module splits and $D\Lambda$ is an injective cogenerator as a left $B$-module as all $\nabla_B(\mathtt{i})$ appear at least once in a direct sum decomposition. Applying duality, it follows that $\Lambda$ is a projective generator as a right $B$-module. 
\end{proof}

We now recall the definition of a directed bocs and the main result of \cite{KKO14}. 

\begin{defn}
Let $\Bbbk$ be an algebraically closed field. A bocs $(B,W)$ is said to be \emphbf{directed} if $B$ is directed, $\overline{W}=\ker\varepsilon$ is a projective bimodule and every indecomposable direct summand $Be_{\mathtt{j}}\otimes_\Bbbk e_{\mathtt{i}}B$ of $\overline{W}$ satisfies $\mathtt{i}<\mathtt{j}$. 
\end{defn}

\begin{thm}
Let $\Bbbk$ be an algebraically closed field. An algebra $\Lambda$ is quasi-hereditary if and only if it is Morita equivalent to the right algebra $R$ of a directed bocs $(B,W)$. In this case, $B$ is an exact Borel subalgebra of $R$.
\end{thm}

In \cite{KKO14} it moreover has been shown that then the category of
representations of $(B,W)$ is equivalent as an exact category to the
category ${\mathcal F}(\Delta)$
of $\Lambda$-modules with standard filtrations. This description has
been used by Bautista, P\'erez and Salmer\'on in \cite{BPS17} to
establish the tame-wild dichotomy for the categories ${\mathcal F}(\Delta)$.

Before proving our main result, we recall the notion of regularity for normal bocses which is part of the reduction algorithm used in the proof of Drozd's tame-wild dichotomy theorem. It is essentially due to Kleiner and Ro\u{\i}ter \cite{KR77}, for the precise formulation see \cite[Proposition 3.11]{KM16}.

\begin{prop}\label{regularisation}
Let $\mathfrak{B}=(B,W)$ be a normal bocs with group-like $\omega$ and projective kernel $\overline{W}$. Define $\partial_0: B\to W$, by $\partial_0(a)=\omega a-a\omega$, and 
assume that there exists $a\in B$ such that $\partial_0(a)=\lambda\psi+\sum_i c_i\psi_ib_i$ with $b_i,c_i\in B$ and $\psi\neq \psi_i$ are generators of $\overline{W}$ and $0\neq \lambda\in \Bbbk$. Then, there is a bocs $\tilde{B}=(\tilde{B},\tilde{W})$ with $\tilde{B}=B/(a)$ and $\tilde{W}=\tilde{B}\otimes_B W/(\psi)\otimes_B \tilde{B}$ such that the following statements hold:
\begin{enumerate}[(i)]
\item There is an equivalence of categories $\modu \mathfrak{B}\cong \modu \tilde{\mathfrak{B}}$. 
\item If $\mathfrak{B}=(B,W)$ is directed, then $\tilde{\mathfrak{B}}=(\tilde{B},\tilde{W})$ is directed. 
\item The right algebra of $\mathfrak{B}$ is Morita equivalent to the right algebra of $\tilde{\mathfrak{B}}$. 
\end{enumerate}
\end{prop}

\begin{defn}
A bocs $\mathfrak{B}$ is said to be  \emphbf{regular} if Proposition \ref{regularisation} cannot be applied anymore, 
that is, no $a$ exists with the required property.
\end{defn}

The following result was stated by Ovsienko in unpublished notes. A proof can be found in \cite[Lemma 5.3]{KM16}.  

\begin{lem}\label{regularityisomorphismext1}
Let $\mathfrak{B}=(B,W)$ be a directed normal bocs with right algebra $R$. Then the following conditions are equivalent:
\begin{enumerate}[(1)]
\item $\mathfrak{B}$ is regular.
\item $\Ext^1_B(\mathbb{L},\mathbb{L})\to \Ext^1_R(\Delta,\Delta)$ is an isomorphism.
\end{enumerate}
\end{lem}

Here, $\mathbb{L}$ is a complete sum of representatives of simple 
$B$-modules and $\Delta$ is a complete sum of representatives of
standard $R$-modules. 
\bigskip

We are now ready to prove the main result of this article.

\begin{thm} \label{maintheorem}
  Let $\Bbbk$ be an algebraically closed field. Then there is a one-to-one correspondence:  
\[\begin{tikzcd}
\{\text{quasi-hereditary algebras  with (normal)  homological exact Borel subalgebra}\}\arrow[<->]{d}{1-1}  \\ 
\{\text{directed (normal) bocses}\},
\end{tikzcd}
\]
which restricts  to a one-to-one correspondence:
\[\begin{tikzcd}
\{\text{quasi-hereditary algebras  with {\em regular} homological exact Borel subalgebra}\}\arrow[<->]{d}{1-1}  \\
\{\text{{\em regular} directed bocses}\}.
\end{tikzcd}
\]
\end{thm}

\begin{proof}
The upwards map 
is given by sending a directed bocs to its right algebra $R=\Hom_B(W,B)^{\op}$. That it is well-defined was proved in \cite{KKO14}. For the map in the other direction note that 
by Lemma \ref{projectivegenerator} the algebra $\Lambda$ is a projective generator for $B$. Thus, $\Hom_{B^{\op}}(\Lambda,B)$ is a $B$-coring with surjective counit. Furthermore, Theorem \ref{kernelprojectivising} shows that $\overline{W}:=\ker\varepsilon$ is projectivising. Since the ground field is assumed to be algebraically closed, in particular perfect, Proposition \ref{perfect} implies that $\overline{W}$ is a projective bimodule and hence is a direct sum of bimodules of the form $Be_\mathtt{i}\otimes_\Bbbk e_\mathtt{j}B$. By the definition of exact Borel subalgebra, $B$ is directed. It remains to prove that, for each summand $Be_\mathtt{i}\otimes_\Bbbk e_\mathtt{j}B$ of $\Hom_{B^{\op}}(\Lambda,B)$, the inequality $\mathtt{i}>\mathtt{j}$ holds. Suppose that this is not the case, i.e. $\mathtt{i}\leq \mathtt{j}$. Note that $\Hom_\Bbbk(L_B(\mathtt{j}),L_B(\mathtt{i}))\cong  L_{B\otimes_\Bbbk B^{\op}}((\mathtt{j},\mathtt{i}))$, the simple module corresponding to the indecomposable projective bimodule $Be_\mathtt{j}\otimes_\Bbbk e_\mathtt{i}B$. Applying $\Hom_{B\otimes_\Bbbk B^{\op}}(-,\Hom_{\Bbbk}(L(\mathtt{i}),L(\mathtt{j})))$ to the exact sequence $0\to \overline{W}\to {W}\to B\to 0$ yields the long exact sequence 
\begin{equation}\label{equation5}
\begin{split}
0\to \Hom_{B\otimes_\Bbbk B^{\op}}(B,\Hom_\Bbbk(L(\mathtt{j}),L(\mathtt{i})))\to 
\Hom_{B\otimes_\Bbbk B^{\op}}(W,\Hom_\Bbbk(L(\mathtt{j}),L(\mathtt{i})))\\\to 
\Hom_{B\otimes_\Bbbk B^{\op}}(\overline{W},\Hom_\Bbbk(L(\mathtt{j}),L(\mathtt{i})))
\to \Ext^1_{B\otimes_\Bbbk B^{\op}}(B,\Hom_\Bbbk(L(\mathtt{j}),L(\mathtt{i})))\\\twoheadrightarrow
\Ext^1_{B\otimes_\Bbbk B^{\op}}(W,\Hom_\Bbbk(L(\mathtt{j}),L(\mathtt{i}))).
\end{split}
\end{equation}
Using that $\Ext^s_{B\otimes_\Bbbk B^{\op}}(B,\Hom_\Bbbk(L(\mathtt{j},L(\mathtt{i})))\cong \Ext^s_B(L(\mathtt{j}),L(\mathtt{i}))$ for $s=0,1$, the fact that 
\[\Ext^s_{B\otimes_\Bbbk B^{\op}}\left(W,\Hom_\Bbbk\left(L(\mathtt{j}),L(\mathtt{i})\right)\right)\cong \Ext^s_\Lambda(\Delta(\mathtt{j}),\Delta(\mathtt{i})),\]
for $s=0,1$, see e.g. \cite[Section 4]{KM16}, and the fact that $B$ is directed, and therefore $\Ext^1_B(L(\mathtt{j}),L(\mathtt{i}))=0$ for $\mathtt{j}\geq \mathtt{i}$, the long exact sequence
\eqref{equation5} can be rewritten to obtain the exact sequence
\begin{equation}
\begin{aligned}
0\longrightarrow \Hom_B(L(\mathtt{j}),L(\mathtt{i})) &\longrightarrow \Hom_\Lambda(\Delta(\mathtt{j}),\Delta(\mathtt{i}))\\
& \longrightarrow \Hom_{B\otimes_\Bbbk B^{\op}}(\overline{W},\Hom_\Bbbk(L(\mathtt{j}),L(\mathtt{i})))\longrightarrow 0\, .
\end{aligned}
\end{equation}
In particular, a summand $Be_\mathtt{i}\otimes_\Bbbk e_\mathtt{j}B$ of $\overline{W}$ gives rise to a  non-isomorphism
from $\Delta(\mathtt{j})$ to $\Delta(\mathtt{i})$.  
As $\Lambda$ is quasi-hereditary, this is impossible for $\mathtt{j}\geq \mathtt{i}$.  
Regularity follows from Lemma \ref{regularityisomorphismext1}. 
The two constructions are inverses to each other by the dual coring theorem \cite[3.7]{Swe75}; see \cite[17.11]{BW03}.
\end{proof}

The following example provides two instances of exact Borel subalgebras which
are not homological, respectively homological but not regular. 

\begin{ex}
\begin{enumerate}[(i)]
\item As explained in \cite[Appendix A.4]{KKO14}, in the case of the algebra  of Example~\ref{ex.quiver},
there is an 
 extension between $\Delta(\mathtt{1})$ and $\Delta(\mathtt{3})$ which does not come from an extension of $L_B(\mathtt{1})$ and $L_B(\mathtt{3})$.
\item As explained before, examples of non-regular exact Borel subalgebras correspond to non-regular bocses. An example was given in \cite[Appendix A.2]{KKO14} where the quasi-hereditary algebra is $\Lambda=\Bbbk\times M_2(\Bbbk)$ and 
\[B=\left\{(a,\begin{pmatrix}a&b\\0&c\end{pmatrix})|a,b,c\in \Bbbk\right\},\]
which is an exact Borel subalgebra coming from a non-regular (normal) bocs. 
\end{enumerate}
\end{ex}

\end{document}